\documentclass[12pt]{amsart}
\usepackage{amssymb}
\usepackage{amsmath, amscd}
\usepackage{amsthm}
\usepackage{comment}
\usepackage[table,xcdraw]{xcolor}
\usepackage{ulem}
\usepackage{tikz}
\usepackage{float}
\usepackage{graphicx}
\usepackage{circuitikz}
\usepackage[colorlinks=true, linkcolor=blue,urlcolor=blue]{hyperref}

\newtheorem{theorem}{Theorem}[section]

\newtheorem{proposition}[theorem]{Proposition}
\newtheorem{lemma}[theorem]{Lemma}

\newtheorem{corollary}[theorem]{Corollary}
\theoremstyle{definition}

\newtheorem{example}[theorem]{Example}
\newtheorem{definition}[theorem]{Definition}

\newtheorem{remark}[theorem]{Remark}

\newtheorem{problem}[theorem]{Problem}

\begin{document}
	
\author[P. Danchev]{Peter Danchev}
\address{Institute of Mathematics and Informatics, Bulgarian Academy of Sciences, 1113 Sofia, Bulgaria}
\email{danchev@math.bas.bg; pvdanchev@yahoo.com}
\author[A. Javan]{Arash Javan}
\address{Department of Mathematics, Tarbiat Modares University, 14115-111 Tehran Jalal AleAhmad Nasr, Iran}
\email{a.darajavan@modares.ac.ir; a.darajavan@gmail.com}
\author[O. Hasanzadeh]{Omid Hasanzadeh}
\address{Department of Mathematics, Tarbiat Modares University, 14115-111 Tehran Jalal AleAhmad Nasr, Iran}
\email{o.hasanzade@modares.ac.ir; hasanzadeomiid@gmail.com}
\author[M. Doostalizadeh]{Mina Doostalizadeh}
\address{Department of Mathematics, Tarbiat Modares University, 14115-111 Tehran Jalal AleAhmad Nasr, Iran}
\email{d\_mina@modares.ac.ir;  m.doostalizadeh@gmail.com}
\author[A. Moussavi]{Ahmad Moussavi}
\address{Department of Mathematics, Tarbiat Modares University, 14115-111 Tehran Jalal AleAhmad Nasr, Iran}
\email{moussavi.a@modares.ac.ir; moussavi.a@gmail.com}

\title[$n$-$\Delta$U rings]{Rings such that, for each unit $u$, $u^n-1$ belongs to the $\Delta(R)$}

\maketitle

%\date{.}

%\begin{document}

%\maketitle

\begin{abstract}
We study in-depth those rings $R$ for which, there exists a fixed $n\geq 1$, such that $u^n-1$ lies in the subring $\Delta(R)$ of $R$ for every unit $u\in R$. We succeeded to describe for any $n\geq 1$ all reduced $\pi$-regular $(2n-1)$-$\Delta$U rings by showing that they satisfy the equation $x^{2n}=x$ as well as to prove that the property of being exchange and clean are tantamount in the class of $(2n-1)$-$\Delta$U rings. These achievements considerably extend results established by Danchev (Rend. Sem. Mat. Univ. Pol. Torino, 2019) and Ko\c{s}an et al. (Hacettepe J. Math. \& Stat., 2020). Some other closely related results of this branch are also established.  	
\end{abstract}

\section{Introduction and Motivation}

In this paper, let $R$ denote an associative ring with identity element, which is {\it not} necessarily commutative. For such a ring $R$, the sets $U(R)$, $Nil(R)$, $C(R)$ and $Id(R)$ represent the set of invertible elements, the set of nilpotent elements, the set of central elements, and the set of idempotent elements in $R$, respectively. Additionally, $J(R)$ denotes the Jacobson radical of $R$. The ring of \( n \times n \) matrices over \( R \) and the ring of \( n \times n \) upper triangular matrices over \( R \) are denoted by \( {\rm M}_n(R) \) and \( {\rm T}_n(R) \), respectively. A ring is termed {\it abelian} if each its idempotent element is central.

The main instrument of the present article plays the set $\Delta(R)$, which was introduced by Lam in \cite [Exercise 4.24]{18} and recently studied by Leroy-Matczuk in \cite{2}. As pointed out by the authors in \cite [Theorem 3 and 6]{2}, $\Delta(R)$ is the largest Jacobson radical's subring of $R$ which is closed with respect to multiplication by all units (quasi-invertible elements) of $R$. Also, $J(R) \subseteq \Delta(R)$. Moreover, $\Delta(R)=J(T)$, where $T$ is the subring of $R$ generated by units of $R$, and the equality $\Delta(R)=J(R)$ holds if, and only if, $\Delta(R)$ is an ideal of $R$. An element $a$ in a ring $R$ is from $\Delta(R)$ if $1-ua$ is invertible for all invertible $u\in R$.

A ring $R$ is said to be {\it $n$-UJ} provided $u-u^n\in J(R)$ for each unit $u$ of $R$, where $n\geq 2$ is a fixed integer; that is, for any $u\in U(R)$, $u^n-1\in J(R)$. This notion was initially introduced by Danchev in \cite{Dannew} on 2019 and after that, hopefully independently, by Ko\c{s}an et al. in \cite{3} on 2020; note that these rings are a common generalization for $n=1$ of the so-termed {\it JU} rings which were firstly defined by Danchev in \cite{Danew} on 2016 and later redefined in \cite{KLM} on 2018 under the name {\it UJ} rings. They showed that for $(2n)$-UJ rings the notions of semi-regular, exchange and clean rings are equivalent.

Likewise, letting $n\ge 2$ be fixed, a ring $R$ is called {\it $n$-UU} if, for any $u\in U(R)$, $u^n-1\in Nil(R)$. This concept was introduced by Danchev (see \cite{4}), and furthermore studied in more details in \cite{19}. In \cite{19} the authors showed that a ring $R$ is ($n-1$)-UU and strongly $\pi$-regular if, and only if, $R$ is strongly $n$-nil-clean (that is, the sum of an $n$-potent and a nilpotent which commute each other).

A ring $R$ is said to be {\it regular} (resp., {\it unit-regular}) in the sense of von Neumann if, for every $a\in R$, there is $x\in R$ (resp., $x\in U(R)$) such that $axa=a$, and $R$ is said to be {\it strongly regular} if, for every $a\in R$, $a\in a^2R$. Recall also that a ring $R$ is {\it exchange} if, for each $a\in R$, there exists $e^2=e\in aR$ such that $1-e\in (1-a)R$, and a ring $R$ is {\it clean} if every element of $R$ is a sum of an idempotent and an unit (cf. \cite{5}). Notice that every clean ring is exchange, but the converse is manifestly {\it not} true in general; however, it is true in the abelian case (see \cite[Proposition 1.8]{5}). In this aspect, a ring $R$ is called {\it semi-regular} provided $R/J(R)$ is regular and idempotents lift modulo $J(R)$. It is well known that semi-regular rings are always exchange, but the opposite is generally untrue (see, for instance, \cite{5}). In 2019, Fatih Karabacak et al. introduced new rings that are a proper expansion of $UJ$ rings. They called these rings $\Delta U$ in \cite{1}, namely a ring $R$ is said to be {\it $\Delta$U} if $1+\Delta(R)=U(R)$.

So, as a possible non-trivial extension of $\Delta$U rings, we introduce the concept of an $n$-$\Delta$U ring. A ring $R$ is called {\it $n$-$\Delta$U} if, for each $u\in U(R)$, $u^n-1\in \Delta(R)$, where $n\geq2$ is a fixed integer. Clearly, all $\Delta$U rings and rings with only two units are $n$-$\Delta$U. Also, every $n$-UJ ring is $n$-$\Delta$U, but the reciprocal implication does {\it not} hold in all generality.

\medskip

Our basic material is organized as follows: In the next section, we examine the behavior of $n$-$\Delta$U rings comparing their crucial properties with these of the $\Delta$U rings (see, for instance, Theorems~\ref{2.7}, \ref{2.13}, \ref{2.16} and \ref{2.22}, respectively). In the third section, we concentrate on the structure of some key extensions of $n$-$\Delta$U ring demonstrating that there is an abundance of their critical properties (see, e.g., Propositions~\ref{3.1}, \ref{3.2}, \ref{3.13}, \ref{3.4}, \ref{3.3}, \ref{3.14}, \ref{3.6}, etc. and Theorem~\ref{newext}).

\section{$n$-$\Delta$U rings}

In this section, we begin by introducing the notion of $n$-$\Delta$U rings and investigate its elementary properties. We now give our main tools.

\begin{definition}\label{2.1}
A ring $R$ is called $n$-$\Delta$U if, for each $u\in U(R)$, $u^n-1\in \Delta(R)$, where $n\geq2$ is a fixed integer.	 \end{definition}

\begin{definition}\label{2.18}
A ring $R$ is called $\pi$-$\Delta$U if, for any $u\in U(R)$, there exists $i\ge 2$ depending on $u$ such that $u^i - 1\in \Delta(R)$.	
\end{definition}

According to the above two definitions, we observe that every $\Delta$U ring is obviously an $n$-$\Delta$U ring and that every $n$-$\Delta$U ring is a $\pi$-$\Delta$U ring. Besides, it is easy to see that if $R$ is a finite $\pi$-$\Delta$U ring, then one can find some number $m\in \mathbb{N}$ such that $R$ is an $m$-$\Delta$U ring.

\medskip

We now arrive at the following construction.

\begin{example}\label{2.2}
Once again, it is clear that $n$-UJ rings are always $n$-$\Delta$U. However, the converse claim is generally invalid. For example, consider the ring \( R = \mathbb F_2\langle x, y \rangle / \langle x^2 \rangle \). Then, one calculates that \( J(R) = \{0\} \), \( \Delta(R) = \mathbb F_2 x + xRx \) and \( U(R) = 1 + \mathbb F_2x + xRx \). Thus, \( R \) is $\Delta$U in view of \cite [Example 2.2]{1} and hence it is $n$-$\Delta$U. But, evidently, \( R \) is {\it not} $n$-UJ.
\end{example}

We continue with the following technicalities.

\begin{proposition}\label{2.4}
Let $R$ be an $n$-$\Delta$U ring, where $n$ is an odd number. Then, $2\in \Delta(R)$.
\end{proposition}

\begin{proof}
Writing $-1=(-1)^n\in 1+\Delta(R)$ whence $-2\in \Delta(R)$, we apply \cite [Lemma 1(2)]{2} to conclude that $2\in \Delta(R)$, as formulated.
\end{proof}

\begin{remark}
The condition "$n$ is an odd number" in Proposition \ref{2.4} is essential. For instance, \(\mathbb{Z}_6\) is a $2$-$\Delta$U ring, but a simple computation shows that $2\notin \Delta(\mathbb{Z}_6)$.
\end{remark}

\begin{proposition}\label{2.17}
Let $R$ be an $n$-$\Delta$U ring and $k \in \mathbb{N}$ such that $n|k$. Then, $R$ is a $k$-$\Delta$U ring.
\end{proposition}

\begin{proof}
Since $R$ is an $n$-$\Delta$U ring, for any $u \in U(R)$ we may write that $u^n = 1 + r$, where $r \in \Delta(R)$. Since $n|k$, there exists an integer $t$ such that $k = tn$. Thus, $$u^k = (u^n)^t = (1 + r)^t = 1 + r',$$ where $r' = (1+r)^t-1$, which is obviously in $\Delta(R)$ because it is a subring of $R$. Therefore, $u^k = 1 + r'$, where $r' \in \Delta(R)$. Hence, $R$ is a $k$-$\Delta$U ring, as stated.
\end{proof}

\begin{proposition}\label{2.19}
A division ring $D$ is $n$-$\Delta$U if, and only if, $u^n=1$ for every $u\in U(D)$.
\end{proposition}

\begin{proof}
It is straightforward by noticing that for any division ring $D$ we have $\Delta(D)=\{0\}$.	
\end{proof}

\begin{lemma}\label{2.20}
Suppose $\mathbb F$ is a field. Then, $\mathbb F$ is $n$-$\Delta$U if, and only if, $\mathbb F$ is finite and $(|\mathbb F|-1) | n$.	
\end{lemma}

\begin{proof}
Let $f(x)=1-x^n \in \mathbb F[x]$. Since $\mathbb F$ is a field, the polynomial $f(x)$ has at most $n$ roots in $\mathbb F^{\ast}$. So, if we suppose $A$ to be the set of all roots of $f$ in $\mathbb F^{\ast}$, we will have $\mathbb F^{\ast}=A$. Consequently, $|\mathbb F^{\ast}|=|A|<n$.
	
On the other hand, as $\mathbb F^{\ast}$ is a cyclic group, there exists $a\in \mathbb F^{\ast}$ such that $\mathbb F^{\ast} =\langle a\rangle$. Since $a^n=1$, we get $o(a)|n$, and hence $n=o(a)q=|\mathbb F^{\ast}|q$. Therefore, $|\mathbb F^{\ast}||n$ and, finally, $(|\mathbb F|-1) | n$, as pursued.\\
The reverse implication is elementary.
\end{proof}

\begin{lemma}\label{2.21}
Let $D$ be a division ring and $n\ge 2$. If $D$ is $n$-$\Delta$U, then $D$ is a finite field and $(|D|-1) | n$.
\end{lemma}

\begin{proof}
Certainly, $\Delta(D)=\{0\}$. So, for any $a \in D$, we have $a^n=1$, whence $a=a^{n+1}$. Furthermore, appealing to the famous Jacobson's Theorem \cite [12.10]{12}, we detect that $D$ must be commutative, and thus a field, as expected.
	
The second part follows at once from Lemma \ref{2.20}.
\end{proof}

\begin{corollary}
If $D$ is a division ring which is $\pi$-$\Delta$U, then $D$ is a field.	
\end{corollary}

\begin{example}
Consider the ring $\mathbb{Z}$. Knowing that $U(\mathbb{Z})= \{1, -1\}$, it is not too hard to see that $\Delta(\mathbb{Z})=\{0\}$. Hence, $\mathbb{Z}$ is an $n$-$\Delta$U. Nevertheless, for an arbitrary prime number $p$, the ring $\mathbb{Z}_{p}$ is {\it not} $n$-$\Delta$U for every $n$ unless $p-1$ divides $n$ by Lemma \ref{2.20}.
\end{example}

\begin{proposition}\label{2.3}
A direct product \(\prod_{i \in I} R_i\) of rings $R_i$ is $n$-$\Delta$U if, and only if, each direct component \(R_i\) is $n$-$\Delta$U.
\end{proposition}

\begin{proof}
As the equalities \(\Delta(\prod_{i \in I} R_i) = \prod_{i \in I} \Delta(R_i)\) and \(U(\prod_{i \in I} R_i) = \prod_{i \in I} U(R_i)\) are fulfilled, the result follows at once.
\end{proof}

\begin{proposition}\label{2.5}
Let \(R\) be an $n$-$\Delta$U ring. If \(T\) is an epimorphic image of \(R\) such that all units of \(T\) lift to units of \(R\), then \(T\) is $n$-$\Delta$U.
\end{proposition}

\begin{proof}
Suppose that \(f: R \rightarrow T\) is a ring epimorphism. Let \(v \in U(T)\). Then, there exists \(u \in U(R)\) such that \(v = f(u)\) and \(u^n = 1 + r \in 1 + \Delta(R)\). Thus, we have \[
v^n = (f(u))^n = f(u^n) = f(1 + r) = f(1) + f(r) = 1 + f(r) \in 1 + \Delta(T),\] as asked for.
\end{proof}

\begin{proposition}\label{2.6}
Let \(R\) be an $n$-$\Delta$U. For any unital subring \(S\) of \(R\), if \(S \cap \Delta(R) \subseteq \Delta(S)\), then \(S\) is an $n$-$\Delta$U ring. In particular, the center of \(R\) is an $n$-$\Delta$U ring.
\end{proposition}

\begin{proof}
Let \(v \in U(S)\) \(\subseteq U(R)\). Since \(R\) is $n$-$\Delta$U, we have \( v^n-1 \in \Delta(R) \cap S \subseteq \Delta(S)\). So, \(S\) is necessarily an $n$-$\Delta$U ring. The rest of the statement follows directly from \cite[Corollary 8]{2}.
\end{proof}

Our first major assertion is the following necessary and sufficient condition.

\begin{theorem}\label{2.7}
Let \(I \subseteq J(R)\) be an ideal of a ring \(R\). Then \(R\) is $n$-$\Delta$U if, and only if, so is \(R/I\).
\end{theorem}

\begin{proof}
Let \(R\) be $n$-$\Delta$U and \(u + I \in U(R/I)\). Then, \(u \in U(R)\) and thus \(u^n = 1 + r\), where \(r \in \Delta(R)\). Now, \((u + I)^n = u^n + I = (1+I)+(r+I)\), where \(r + I \in \Delta(R)/I = \Delta(R/I)\) in virtue of \cite [Proposition 6]{2}.

Conversely, let \(R/I\) is $n$-$\Delta$U and \(u \in U(R)\). Then, \(u + I \in U(R/I)\) whence \((u + I)^n = (1 + I) + (r + I)\), where \(r + I \in \Delta(R/I)\). Thus, \(u^n + I = (1 + r) + I\) and so \(u^n - (1 + r) \in I \subseteq J(R) \subseteq \Delta(R)\). Therefore, \(u^n = 1 + r^\prime\), where \(r^\prime \in \Delta(R)\). Hence, \(R\) is $n$-$\Delta$U, as required.
\end{proof}

As an automatic consequence, we extract:

\begin{corollary}\label{2.8}
A ring \(R\) is $n$-$\Delta$U if, and only if, \(R/J(R)\) is $n$-$\Delta$U.
\end{corollary}

We next proceed by proving the following structural affirmations.

\begin{proposition}\label{2.9}
Let \(R\) be an $n$-$\Delta$U (resp., a $\pi$-$\Delta$U) ring and let \(e\) be an idempotent of \(R\). Then, \(eRe\) is an $n$-$\Delta$U (resp., a $\pi$-$\Delta$U) ring.
\end{proposition}

\begin{proof}
Let $u \in U(eRe)$. Thus, $u + (1-e) \in U(R)$. By hypothesis, $$(u + (1-e))^n = u^n + (1 - e) = 1 + r \in 1 + \Delta(R).$$ So, we have $u^n - e \in \Delta(R)$. Now, we show that $u^n - e \in \Delta(eRe)$. Let $v$ be an arbitrary unit of $eRe$. Apparently, $v + 1 - e \in U(R)$. Note that $u^n - e \in \Delta(R)$ gives us that $u^n - e + v + 1 - e \in U(R)$ utilizing the definition of $\Delta(R)$. Taking $u^n - e + v + 1 - e = t \in U(R)$, one checks that $$et=te=ete=u^n-e+v,$$ and so $ete \in U(eRe)$. It now follows that $u^n-e+U(eRe) \subseteq U(Re)$. Then, we deduce $u^n - e \in \Delta(eRe)$ implying $u^n\in e+\Delta(eRe)$ which yields that the corner ring $eRe$ is an $n$-$\Delta$U ring, as wanted.\\
The case of $\pi$-$\Delta$U rings is quite similar, so we omit the arguments.
\end{proof}

\begin{proposition}\label{2.10}
For any ring \( R \neq \{0\} \) and any integer \( n \geq 2 \), the ring \( M_n(R) \) is not a $(2k-1)$-$\Delta$U ring whenever $k\geq 1$.
\end{proposition}

\begin{proof}
Since it is long known that \( M_2(R) \) is isomorphic to a corner ring of \( M_n(R) \) for \( n \geq 2 \), it suffices to show that \( M_2(R) \) is not a $(2k-1)$-$\Delta$U ring bearing in mind Proposition \ref{2.9}. To this goal, consider the matrix
\[
A = \begin{pmatrix} 0 & -1 \\ 1 & 0 \end{pmatrix} \in U(M_2(R)).
\]
Thus, $A^{2k-1}=A$ or $A^{2k-1}=-A$. Now, let $M_2(R)$ be $(2k-1)$-$\Delta$U. If firstly $A^{2k-1}=A$, then we conclude that
\[
B:=A-I=\begin{pmatrix} -1 & -1 \\ 1 & -1 \end{pmatrix}\in \Delta(M_2(R)).
\]
But, we know that $B$ is a unit. So, utilizing \cite [Lemma 1]{2}, we infer that $BB^{-1} \in \Delta(M_2(R))$ and hence $I\in \Delta(M_2(R))$. This, however, is an obvious contradiction.

If now $A^{2k-1}=-A$, it can be concluded that $I\in \Delta(M_2(R))$ and again this is a contraposition. So, \( M_2(R) \) is really not a $(2k-1)$-$\Delta$U ring, as desired.
\end{proof}

\begin{example}
Consider the matrix ring $R=M_2(\mathbb{Z}_2)$. We have
\[
U(R)={{\begin{pmatrix} 1 & 0 \\ 0 & 1 \end{pmatrix}, \begin{pmatrix} 0 & 1 \\ 1 & 0 \end{pmatrix}, \begin{pmatrix} 0 & 1 \\ 1 & 1 \end{pmatrix}, \begin{pmatrix} 1 & 0 \\ 1 & 1 \end{pmatrix}, \begin{pmatrix} 1 & 1 \\ 0 & 1 \end{pmatrix}, \begin{pmatrix} 1 & 1 \\ 1 & 0 \end{pmatrix}.}}
\]
With a simple calculation at hand, we may derive that, for any $u\in U(R)$, $u^6-1\in \Delta(R)$. So, $R$ is a $6$-$\Delta$U ring. In general, $M_n(R)$ ($n \geq 2$) is {\it not} $n$-$\Delta$U if $n$ is an odd number. However, this observation does {\it not} hold in general for even values of $n$.
\end{example}

Let us now recollect that a set $\{e_{ij} : 1 \le i, j \le n\}$ of non-zero elements of $R$ is said to be a system of $n^2$ {\it matrix units} if $e_{ij}e_{st} = \delta_{js}e_{it}$, where $\delta_{jj} = 1$ and $\delta_{js} = 0$ for $j \neq s$. In this case, $e := \sum_{i=1}^{n} e_{ii}$ is an idempotent of $R$ and $eRe \cong M_n(S)$, where $$S = \{r \in eRe : re_{ij} = e_{ij}r~~\textrm{for all}~~ i, j = 1, 2, . . . , n\}.$$
Recall also that a ring $R$ is said to be {\it Dedekind-finite} provided $ab=1$ implies $ba=1$ for any two $a,b\in R$. In other words, all one-sided inverse elements in the ring must be two-sided.

\medskip

We are now prepared to establish the following.

\begin{proposition}\label{2.11}
Every $(2k-1)$-$\Delta$U ring is Dedekind-finite, provided $k\geq 1$.
\end{proposition}

\begin{proof}
If we assume the contrary that $R$ is {\it not} a Dedekind-finite ring, then there exist elements $a, b \in R$ such that $ab = 1$ but $ba \neq 1$. Assuming $e_{ij} = a^i(1-ba)b^j$ and $e =\sum_{i=1}^{n}e_{ii}$, there exists a non-zero ring $S$ such that $eRe \cong M_n(S)$. However, owing to Proposition \ref{2.9}, $eRe$ is a $(2k-1)$-$\Delta$U ring, so $M_n(S)$ has to be a $(2k-1)$-$\Delta$U ring too, which contradicts Proposition \ref{2.10}, as expected.
\end{proof}

Recall that a ring $R$ is said to be {\it semi-local} if $R/J(R)$ is a left artinian ring or, equivalently, if $R/J(R)$ is a semi-simple ring.

\begin{proposition}
Let $R$ be a ring and $n\geq 1$. Then, the following two conditions are equivalent for a semi-local ring:
\begin{enumerate}
\item
$R$ is a $(2n-1)$-$\Delta$U ring.
\item
$R/J(R) \cong \prod_{i=1}^{m} \mathbb F_{p^{k_i}}$, where $(p^{k_i}-1) | n$ and $\mathbb F_{p^{k_i}}$ is a field with $p^{k_i}$ elements.
\end{enumerate}
\end{proposition}

\begin{proof}
(i) $\Longrightarrow$ (ii). Since $R$ is semi-local, $R/J(R)$ is semi-simple, so we have $$R/J(R) \cong \prod_{i=1}^{m} {\rm M}_{n_i}(D_i),$$ where each $D_i$ is a division ring. Then, employing Corollary \ref{2.8} and Proposition \ref{2.10}, we deduce that $R/J(R) \cong \prod_{i=1}^{m} D_i.$ On the other hand, invoking Lemma \ref{2.21}, we derive that $D_i\cong \mathbb F_{p^{k_i}}$, where $p^{k_i}-1$ divides $n$, as claimed.\\
(ii) $\Longrightarrow$ (i). According to Lemma \ref{2.20}, we conclude that every $\mathbb F_{p^{k_i}}$ is $(2n-1)$-$\Delta$U for all $i$. Then, taking into account Proposition \ref{2.3}, we receive that $\prod_{i=1}^{m} \mathbb F_{p^{k_i}}$ is $(2n-1)$-$\Delta$U and hence $R/J(R)$ is $(2n-1)$-$\Delta$U. Thus, $R$ is a $(2n-1)$-$\Delta$U ring in accordance with Corollary \ref{2.8}, as asserted.
\end{proof}

\begin{lemma}\label{2.12}
Let \(R\) be a $(2n-1)$-$\Delta$U ring for some $n\geq 1$. If \(J(R) = \{0\}\) and every non-zero right ideal of \(R\) contains a non-zero idempotent, then \(R\) is reduced.
\end{lemma}

\begin{proof}
Suppose the reverse that \(R\) is {\it not} reduced. Then, there exists a non-zero element \(a \in R\) such that \(a^2 = 0\). Referring to \cite [Theorem 2.1]{9}, there is an idempotent \(e \in RaR\) such that \(eRe \cong M_2(T)\) for some non-trivial ring \(T\). However, thanks to Proposition \ref{2.9}, $eRe$ is a $(2n-1)$-$\Delta$U ring and hence $M_2(T)$ is a $(2n-1)$-$\Delta$U ring as well. This, in turn, contradicts Proposition \ref{2.10}, as expected.
\end{proof}

It is principally known that a ring $R$ is called {\it $\pi$-regular} if, for each $a$ in $R$, $a^n\in a^nRa^n$ for some integer $n$. So, regular rings are always $\pi$-regular. Also, a ring $R$ is said to be {\it strongly $\pi$-regular} provided that, for any $a\in R$, there exists $n$ depending on $a$ such that $a^n\in a^{n+1}R$.

\medskip

Our second main statement is the following.

\begin{theorem}\label{2.13}
Let \(R\) be a ring and $n\geq 1$. The following three items are equivalent:
\begin{enumerate}
\item
\(R\) is a regular $(2n-1)$-$\Delta$U ring.
\item
\(R\) is a \(\pi\)-regular reduced $(2n-1)$-$\Delta$U ring.
\item
\(R\) has the identity \(x^{2n} = x\).
\end{enumerate}
\end{theorem}

\begin{proof}
(i) $\Rightarrow$ (ii). Since \(R\) is regular, \(J(R) = \{0\}\) and thus every non-zero right ideal contains a non-zero idempotent. So, Lemma \ref{2.12} applies to get that \(R\) is reduced. Moreover, every regular ring is known to be \(\pi\)-regular and so the implication follows immediately, as promised.

(ii) $\Rightarrow$ (iii). Notice that reduced rings are always abelian, so \(R\) is abelian regular by \cite [Theorem 3]{10} and hence it is strongly regular. Then, \(R\) is unit-regular and so \(\Delta(R) = \{0\}\) by \cite [Corollary 16]{2}. Thus, we have \(Nil(R) = J(R) = \Delta(R) = \{0\}\).

On the other hand, one observes that \(R\) is strongly \(\pi\)-regular. Let \(x \in R\). In view of \cite [Proposition 2.5]{11}, there is an idempotent \(e \in R\) and a unit \(u \in R\) such that \(x = e + u\), \(ex = xe \in Nil(R) = \{0\}\). So, it must be that \[x = x - xe = x(1-e) = u(1-e) = (1-e)u.\] But, since \(R\) is a $(2n-1)$-$\Delta$U ring, \(u^{2n-1} = 1\). It follows now that \[x^{2n-1}=((1-e)u)^{2n-1}=u^{2n-1}(1-e)^{2n-1}=(1-e).\] Hence, \(x = x(1-e) = x.x^{2n-1} = x^{2n}\), and we are done.

(iii) $\Rightarrow$ (i). It is trivial that \(R\) is regular. Let \(u \in U(R)\). Then, we have \(u^{2n} = u\) forcing that \(u^{2n-1} = 1\) and thus \(R\) is a $(2n-1)$-$\Delta$U ring, as promised.
\end{proof}

We now can record the following interesting consequence.

\begin{corollary}\label{2.14}
Suppose $n\geq 1$. The following four conditions are equivalent for a ring \(R\):
\begin{enumerate}
\item
\(R\) is a regular $(2n-1)$-$\Delta$U ring.
\item
\(R\) is a strongly regular $(2n-1)$-$\Delta$U ring.
\item
\(R\) is a unit-regular $(2n-1)$-$\Delta$U ring.
\item
\(R\) has the identity \(x^{2n} = x\).
\end{enumerate}
\end{corollary}

\begin{proof}
(i) $\Rightarrow$ (ii). In virtue of Lemma \ref{2.12}, \(R\) is reduced and hence abelian. Then, \(R\) is strongly regular.

(ii) $\Rightarrow$ (iii). This is pretty obvious, so we leave out the argumentation.

(iii) $\Rightarrow$ (iv). Let \(x \in R\). Then, \(x = u e\) for some \(u \in U(R)\) and \(e \in Id(R)\). We know that every unit-regular ring is by definition regular, so \(R\) is regular $(2n-1)$-$\Delta$U whence \(R\) is abelian. On the other hand, \cite [Corollary 16]{2} leads us to \(\Delta(R) = \{0\}\). Therefore, for any \(u \in U(R)\), we have \(u^{2n-1} = 1\) which means that \(x^{2n-1} = u^{2n-1}e^{2n-1} = e\). So, we detect that $x^{2n}=x$, as required.

(iv) $\Rightarrow$ (i). It is clear by a direct appeal to Theorem \ref{2.13}.
\end{proof}

Let us recall that a ring $R$ is called {\it semi-potent} if every one-sided ideal {\it not} contained in $J(R)$ contains a non-zero idempotent.

\medskip

The next difficult question arises quite logical.

\begin{problem}
Characterize semi-potent $n$-$\Delta$U rings for an arbitrary $n\geq 1$.	
\end{problem}

The following technical claim is useful.

\begin{proposition}\label{2.15}
Suppose $k\geq 1$. Then, a ring \(R\) is $\Delta$U if, and only if,
\begin{enumerate}
\item
\(2 \in \Delta(R)\),
\item
\(R\) is a $2^k$-$\Delta$U ring,
\item
If, for every \(x \in R\), \(x^{2^k} \in \Delta(R)\), then \(x \in \Delta(R)\).
\end{enumerate}
\end{proposition}

\begin{proof}
"\(\Rightarrow\)". As \(R\) is a $\Delta$U ring, then \(-1 = 1 + r\) for some \(r \in \Delta(R)\). This implies that \(-2 \in \Delta(R)\) and so \(2 \in \Delta(R)\). Besides, every $\Delta$U ring is $2^k$-$\Delta$U. Now, the asked result follows from \cite [Proposition 2.4(3)]{1}.

"\(\Leftarrow\)". Let \(u \in U(R)\). By (ii), we have $u^{2^k}\in 1+\Delta(R)$ and hence, combining \cite [Theorem 3(2) and Lemma 1(3)]{2} with (i), we conclude that $(u-1)^{2^k}=1+u^{2k}+r$ for some $r\in \Delta(R)$. So, $(u-1)^{2^k}\in\Delta(R)$. Thus, with the help of (iii), we conclude that $u-1\in \Delta(R)$, which ensures that $R$ is a $\Delta U$-ring, as required.
\end{proof}

We now come to the next two pivotal assertions.

\begin{theorem}\label{2.16}
Let \( R \) be a $(2n-1)$-$\Delta$U ring. Then, the following two points are equivalent:
\begin{enumerate}
\item
\( R \) is an exchange ring.
\item
\( R \) is a clean ring.
\end{enumerate}
\end{theorem}

\begin{proof}
(ii) $\Rightarrow$ (i). This is obvious, because each clean ring is always exchange.

(i) $\Rightarrow$ (ii). If $R$ is simultaneously exchange and $(2n-1)$-$\Delta$U, then $R$ is reduced thanking to Lemma \ref{2.12}, and hence it is abelian. Therefore, $R$ is abelian exchange, so it is clean.
\end{proof}

\begin{theorem}\label{2.22}
Let \( R \) be a ($2^k$-1)-$\Delta$U ring for some $k\geq 1$. Then, the following three statements are equivalent:
\begin{enumerate}
\item
\( R \) is a semi-regular ring.	
\item
\( R \) is an exchange ring.
\item
\( R \) is a clean ring.
\end{enumerate}
\end{theorem}

\begin{proof}
Observe that (ii) and (iii) are equivalent employing Theorem \ref{2.16}.\\	
(i) $\Rightarrow$ (ii). This is obvious, since every semi-regular ring is always exchange.\\
(iii) $\Rightarrow$ (i). First, we show that $2 \in J(R)$. To this end, Proposition \ref{2.4} assures that \( 2 \in \Delta(R) \). Let \( r \in R \) and \( r = e + u \) be a clean decomposition for \( r \). We know that \( 2e - 1 \in U(R) \) and hence \( (2e - 1)={(2e - 1)}^{2^k-1} \in 1 + \Delta(R) \), so that \( 2e \in \Delta(R) \). Thus, \( 2r = 2e + 2u \in \Delta(R) + \Delta(R) \subseteq \Delta(R) \).
So, \( 1 - 2r \in U(R) \) and hence \( 2 \in J(R) \), as claimed.

On the other hand, \( r^{2^k} = e + 2f + u^{2^k} \), where \( f \in R \). So,
\[
r - r^{2^k} = (e + u) - (e + 2f + u^{2^k}) = (e + u) - (e + 2f + u(u^{2^k-1})) = (e + u) - (e + 2f + u + d)
, \]
whence
\[
r - r^{2^k}= -(2f + d) \in \Delta(R),
\]
where \( d \in \Delta(R) \). Consider now $\overline{R}=R/J(R)$, where $\overline{R}$ is reduced and so abelian enabled via Lemma \ref{2.12}.

Now, we prove that \( \Delta(R) = J(R) \). Letting $d \in \Delta(R)$ and $e \in \text{Id}(R)$, we have \( 1 - e d = f + u \), where \( f \in \text{Id}(R) \) and \( u \in U(R) \). So, \( \overline{1} - \overline{e} \overline{d} = \overline{f} + \overline{u} \) and multiplying by the expression \( \overline{(1 - e)} \) on the left the previous equality, we derive that \( \overline{(1 - e)} = \overline{(1 - e)} \overline{f} + \overline{(1 - e)} \overline{u} \). Then, one inspects that

\[
\overline{(1 - e)} \hspace{1mm} \overline{(1 - f)} = \overline{(1 - e)}\overline{u} \in U(\overline{(1 - e)} \hspace{1mm} \overline{R} \hspace{1mm}\overline{(1 - e)}) \cap \text{Id}(\overline{(1 - e)} \hspace{1mm} \overline{R} \hspace{1mm} \overline{(1 - e)}).
\]
Consequently, \( \overline{(1 - e)} \hspace{1mm} \overline{(1 - f)} = \overline{(1 - e)} \), so again using this trick for the expression \( \overline{f} \) on the right of the previous equality, we deduce that \( \overline{(1 - e)} \overline{f} = \overline{0} \), so that \( \overline{f} = \overline{e} \overline{f} \in \overline{e} \overline{R} \overline{e} \).

Furthermore, if we multiply the equation \( \overline{1} - \overline{e} \overline{d} = \overline{f} + \overline{u} \) by \( \overline{e} \) on the left, we will have
\(\overline{e} - \overline{e} \overline{d} = \overline{e} \overline{f} + \overline{e} \overline{u} = \overline{f} + \overline{e} \overline{u}.\)
Hence,
\[
\overline{e} - \overline{f} = \overline{e} (\overline{u} + \overline{d}) \in U(\overline{e} \overline{R} \overline{e}) \cap \text{Id}(\overline{e} \overline{R} \overline{e}),
\]
and so \( \overline{e} - \overline{f} = \overline{e} \) concluding that \( \overline{f} = \overline{0} \). Then, \( f \in J(R) \cap \text{Id}(R) = \{ 0 \} \).
Thus, \( f = 0 \) and hence \( 1 - e d \in U(R) \).

On the other side,
\[
1 - r d = 1 - e d - u d \in U(R) + \Delta(R) \subseteq U(R),
\]
and we infer that \( d \in J(R) \). Hence, \( r - r^{2^k} \in J(R) \).
Thus, the quotient \( \frac{R}{J(R)} \) is regular and also idempotents lift modulo $J(R)$, because by hypothesis $R$ is a clean ring, whence finally \( R \) is a semi-regular ring, as required.
\end{proof}

\section{Some Extensions of $n$-$\Delta U$ Rings}

As usual, we say that $B$ is an unital subring of a ring $A$ if $\emptyset \neq B\subseteq A$ and, for any $x,y\in B$, the relations $x-y$, $xy\in B$ and $1_{A}\in B$ hold. Let $A$ be a ring and let $B$ an unital subring of $A$, we denote by $R[A,B]$ the set $$\lbrace (a_{1},\ldots ,a_{n},b,b,\ldots ): a_{i}\in A,b\in B,1\leq i\leq n \rbrace.$$ Then, a routine check establishes that $R[A,B]$ forms a ring under the usual component-wise addition and multiplication. The ring $R[A,B]$ is called the {\it tail ring extension}.

\medskip

We start our considerations here with the following helpful statement.

\begin{proposition}\label{3.1}
\(R[A, B]\) is an $n$-$\Delta$U ring if, and only if, both \(A\) and \(B\) are $n$-$\Delta$U rings.
\end{proposition}

\begin{proof}
Suppose \(R[A,B]\) is an $n$-$\Delta U$ rings. Firstly, we prove that $A$ is an $n$-$\Delta$U ring. Let $u\in U(A)$. Then, $\bar{u} = (u,1, 1, \ldots)\in U(R[A,B])$. By hypothesis, we have $(u^n - 1, 0, 0, \ldots) \in \Delta(R[A,B])$, so $(u^n-1, 0, 0, \ldots) + U(R[A,B]) \subseteq U(R[A,B])$. Thus, for all $v \in U(A)$, $$(u^n -1 + v,1, 1, \ldots) = (u^2 - 1,0, 0, \ldots) + (v, 1, 1, \ldots) \in U(R[A,B]).$$
Hence, $u^n -1 + v \in U(A)$, which insures that $u^n -1 \in \Delta(A)$. Now, we show that $B$ is an $n$-$\Delta$U ring. To this target, choose $v \in U(B)$. Then, $(1, \ldots, 1, 1, v, v, \dots)\in U(R[A,B])$. By hypothesis, $(0, \ldots, 0, v^n - 1, v^n - 1, \dots) \in \Delta(R[A,B])$, so $$(0, \ldots, 0, v^n - 1, v^n -1, \dots) + U(R[A,B]) \subseteq U(R[A,B]).$$ Thus, for all $u \in U(B)$, $$(1,1, \dots, v^n -1 + u, v^n -1 + u, \dots) \in U(R[A,B]).$$ We have $v^n - 1 + u \in U(B)$ and hence $v^n -1 \in \Delta(B)$, as required. The case of $A$ is treated absolutely analogously, so we remove the arguments.

Conversely, assume that $A$ and $B$ are both $n$-$\Delta$U rings. Let $$\bar{u} = (u_1, u_2, \ldots, u_t, v, v, \ldots) \in U(R[A,B]),$$ where $u_i \in U(A)$ and $v \in U(B) \subseteq U(A)$. We must show that $\bar{u}^{n} - 1 + U(R[A,B]) \subseteq U(R[A,B])$. In fact, for all $\bar{a} = (a_1, \ldots, a_m, b, b, \ldots) \in U(R[A,B])$ with $a_i \in U(A)$ and $b \in U(B) \subseteq U(A)$, take $z = \max \{m, t\}$. Then, we obtain $$\bar{u}^n - 1 + \bar{a} = (u_1^n-1 + a_1, \dots, u_z^2 - 1 + a_z, v^n - 1 + b, v^n - 1 + b, \dots).$$ Note that $u_i^n - 1 + a_i \in U(A)$ for all $1 \leq i \leq z$ and $v^n - 1 + b \in U(B) \subseteq U(A)$. We, thereby, deduce that $\bar{u}^n - 1 + \bar{a} \in U(R[A,B])$. Thus, $\bar{u}^n - 1\in \Delta(R[A,B])$ and $\bar{u}^n\in 1 + \Delta(R[A,B])$. This unambiguously enables us that $R[A,B]$ is an $n$-$\Delta$U ring, as asserted.	
\end{proof}

Let $R$ be a ring and suppose that $\alpha : R \to R$ is a ring endomorphism. Traditionally, $R[[x; \alpha]]$ denotes the ring of {\it skew formal power series} over $R$; that is, all formal power series in $x$ having coefficients from $R$ with multiplication defined by $xr = \alpha(r)x$ for all $r \in R$. In particular, $R[[x]] = R[[x; 1_R]]$ is the ring of {\it formal power series} over $R$.

\begin{proposition}\label{3.2}
The ring $R[[x; \alpha]]$ is $n$-$\Delta$U if, and only if, so is $R$.
\end{proposition}

\begin{proof}
Consider $I= R[[x; \alpha]]x$. Then, a plain check gives that $I$ is an ideal of $R[[x; \alpha]]$. Note that $J(R[[x; \alpha]])=J(R)+I$, so $I\subseteq J(R[[x; \alpha]])$. Since $R[[x; \alpha]]/I\cong R$, the result follows at once exploiting Theorem \ref{2.7}.
\end{proof}

As an automatic consequence, we yield:

\begin{corollary}
The ring $R[[x]]$ is $n$-$\Delta$U if, and only if, so is $R$.
\end{corollary}

Let $R$ be a ring and suppose that $\alpha : R \to R$ is a ring endomorphism. Standardly, $R[x; \alpha]$ denotes the ring of {\it skew polynomials} over $R$ with multiplication defined by $xr = \alpha(r)x$ for all $r \in R$. In particular, $R[x] = R[x; 1_R]$ is the ring of {\it polynomials} over $R$. For an endomorphism $\alpha$ of a ring $R$, $R$ is called {\it $\alpha$-compatible} if, for any $a,b\in R$, $ab=0\Longleftrightarrow a\alpha (b)=0$, as in this case $\alpha$ is evidently injective.

Let $Nil_{*}(R)$ denote the {\it prime radical} (or, in other terms, the {\it lower nil-radical}) of a ring $R$, i.e., the intersection of all prime ideals of $R$. We know that $Nil_{*}(R)$ is a nil-ideal of $R$. It is long known that a ring $R$ is called {\it $2$-primal} if its lower nil-radical $Nil_{*}(R)$ consists precisely of all the nilpotent elements of $R$. For instance, it is well known that both reduced and commutative rings are $2$-primal.

\begin{proposition}\label{3.13}
Let \( R \) be a $2$-primal and \(\alpha\)-compatible ring. Then, the equality \(\Delta(R[x, \alpha]) = \Delta(R) + Nil_{*}(R[x, \alpha])x\) is valid.
\end{proposition}

\begin{proof}
Assuming \( f = \sum_{i=0}^{n} a_i x^i \in \Delta(R[x, \alpha]) \), then, for every \( u \in U(R) \), we have that \( 1 - uf \in U(R[x, \alpha]) \). Thus, \cite [Corollary 2.14]{21} employs to get that \( 1 - ua_0 \in U(R) \) and, for every \( 1 \le  i \le  n \), the relation \( ua_i \in Nil_{*}(R) \) is true. Since \( Nil_{*}(R) \) is an ideal, it must be that \( a_0 \in \Delta(R) \) and, for every \( 1 \le  i \le  n \), the relation \( a_i \in Nil_{*}(R) \) holds. But, as \( R \) is a $2$-primal ring, \cite [Lemma 2.2]{21} is applicable to conclude that \( Nil_{*}(R)[x,\alpha] = Nil_{*}(R[x, \alpha]) \), as required.
	
Reciprocally, assume \( f \in \Delta(R) + Nil_{*}(R[x, \alpha])x \) and \( u \in U(R[x,\alpha]) \). Then, owing to \cite [Corollary 2.14]{21}, we have \( u \in U(R) + Nil_{*}(R[x, \alpha])x \). Since \( R \) is a $2$-primal ring, one has that \[ 1 - uf \in U(R) + Nil_{*}(R[x, \alpha])x \subseteq U(R[x, \alpha]), \] and thus \( f \in \Delta(R[x, \alpha]) \), as needed.
\end{proof}

We are now in a position to establish the following criterion.

\begin{theorem}\label{newext}
Let \( R \) be a $2$-primal ring and \( \alpha \) an endomorphism of \( R \) such that \( R \) is \( \alpha \)-compatible. The following are equivalent:
\begin{enumerate}
\item
\( R[x; \alpha] \) is an $n$-$\Delta$U ring.
\item
\( R \) is an $n$-$\Delta$U ring.
\end{enumerate}
\end{theorem}

\begin{proof}
(ii) $\Rightarrow$ (i). Let $f = \sum_{i=0}^{n} u_i x^i \in U(R[x, \alpha])$, so in view of \cite [Corollary 2.14]{21} one arrives at \( u_0 \in U(R) \) and \( u_i \in Nil(R) \) for each \( i \geq 1 \). Then, by hypothesis, \(1- u_0^n\in \Delta(R) \). Therefore, with \cite [Corollary 2.14]{21} at hand, there exists $g\in {Nil}_*(R)[x; \alpha]$ such that
\[
f^n=u_0^n+gx\in 1+\Delta(R)+Nil_{*}(R[x, \alpha])x,
\]
and hence with the aid of Proposition \ref{3.13} we obtain
\[
f^n\in 1+\Delta(R[x; \alpha]).
\]
(i) $\Rightarrow$ (ii). Let \( u \in U(R) \subseteq U(R[x; \alpha]) \). Hence, \[ u^n \in 1 + \Delta(R[x; \alpha]) = 1 + \Delta(R) + Nil_{*}(R[x, \alpha])x .\] Thus, we have \( u^n \in 1 + \Delta(R) \) whence \( R \) is an $n$-$\Delta$U ring, as wanted.
\end{proof}

As a valuable consequence, we arrive at the following.

\begin{corollary}
Let \( R \) be a \(2\)-primal ring. Then, the following are equivalent:
\begin{enumerate}
\item
\( R[x] \) is an $n$-$\Delta$U ring.
\item
\( R \) is an $n$-$\Delta$U ring.
\end{enumerate}
\end{corollary}

Let $R$ be a ring and $M$ a bi-module over $R$. The {\it trivial extension} of $R$ and $M$ is stated as
\[ T(R, M) = \{(r, m) : r \in R \text{ and } m \in M\}, \]
with addition defined component-wise and multiplication defined by
\[ (r, m)(s, n) = (rs, rn + ms). \]
One knows that the trivial extension $T(R, M)$ is isomorphic to the subring
\[ \left\{ \begin{pmatrix} r & m \\ 0 & r \end{pmatrix} : r \in R \text{ and } m \in M \right\} \]
of the formal $2 \times 2$ matrix ring $\begin{pmatrix} R & M \\ 0 & R \end{pmatrix}$. We also notice that the set of units of the trivial extension $T(R, M)$ is precisely
\[ U(T(R, M)) = T(U(R), M). \]
Also, by \cite {1}, one may exactly write that
\[ \Delta(T(R, M)) = T(\Delta(R), M). \]

We are now ready to prove the following.

\begin{proposition}\label{3.4}
Let $R$ be a ring and $M$ a bi-module over $R$. Then, the following hold:
\begin{enumerate}
\item
The trivial extension ${\rm T}(R, M)$ is an $n$-$\Delta U$ ring if, and only if, $R$ is an $n$-$\Delta U$ ring.
\item
The upper triangular matrix ring $T_n(R)$ is an $n$-$\Delta U$ if, and only if, $R$ is an $n$-$\Delta U$ ring.
\end{enumerate}
\end{proposition}

\begin{proof}
\begin{enumerate}
\item
Set $A={\rm T}(R, M)$ and consider the ideal $I:={\rm T}(0, M)$. Then, one finds that $I\subseteq J(A)$ such that $\dfrac{A}{I} \cong R$. So, the result follows directly from Theorem \ref{2.7}.
\item
Let $I=\left\lbrace (a_{ij})\in T_{n}(R)\, | \, a_{ii}=0\right\rbrace$. Then, one establishes that $I\subseteq J(T_{n}(R))$ with $T_{n}(R)/I\cong R^n$. Therefore, the desired result follows from a plain combination of Theorem \ref{2.7} and Proposition \ref{2.3}.
\end{enumerate}
\end{proof}

Let $\alpha$ be an endomorphism of $R$ and $n$ a positive integer. It was defined by Nasr-Isfahani in \cite{17} the {\it skew triangular matrix ring} like this:

$${\rm T}_{n}(R,\alpha )=\left\{ \left. \begin{pmatrix}
	a_{0} & a_{1} & a_{2} & \cdots & a_{n-1} \\
	0 & a_{0} & a_{1} & \cdots & a_{n-2} \\
	0 & 0 & a_{0} & \cdots & a_{n-3} \\
	\ddots & \ddots & \ddots & \vdots & \ddots \\
	0 & 0 & 0 & \cdots & a_{0}
\end{pmatrix} \right| a_{i}\in R \right\}$$

\medskip

with addition point-wise and multiplication given by:

\medskip

\begin{align*}
	&\begin{pmatrix}
		a_{0} & a_{1} & a_{2} & \cdots & a_{n-1} \\
		0 & a_{0} & a_{1} & \cdots & a_{n-2} \\
		0 & 0 & a_{0} & \cdots & a_{n-3} \\
		\ddots & \ddots & \ddots & \vdots & \ddots \\
		0 & 0 & 0 & \cdots & a_{0}
	\end{pmatrix}\begin{pmatrix}
		b_{0} & b_{1} & b_{2} & \cdots & b_{n-1} \\
		0 & b_{0} & b_{1} & \cdots & b_{n-2} \\
		0 & 0 & b_{0} & \cdots & b_{n-3} \\
		\ddots & \ddots & \ddots & \vdots & \ddots \\
		0 & 0 & 0 & \cdots & b_{0}
	\end{pmatrix}  =\\
	& \begin{pmatrix}
		c_{0} & c_{1} & c_{2} & \cdots & c_{n-1} \\
		0 & c_{0} & c_{1} & \cdots & c_{n-2} \\
		0 & 0 & c_{0} & \cdots & c_{n-3} \\
		\ddots & \ddots & \ddots & \vdots & \ddots \\
		0 & 0 & 0 & \cdots & c_{0}
	\end{pmatrix},
\end{align*}
where $$c_{i}=a_{0}\alpha^{0}(b_{i})+a_{1}\alpha^{1}(b_{i-1})+\cdots +a_{i}\alpha^{i}(b_{0}),~~ 1\leq i\leq n-1
.$$

\medskip

We denote the elements of ${\rm T}_{n}(R, \alpha)$ by $(a_{0},a_{1},\ldots , a_{n-1})$. If $\alpha $ is the identity endomorphism, then one easily checks that ${\rm T}_{n}(R,\alpha )$ is a subring of the {\it upper triangular matrix ring} ${\rm T}_{n}(R)$.

\medskip

All of the mentioned above guarantee the truthfulness of the following statement.

\begin{proposition}\label{3.3}
Let $R$ be a ring and $k\geq 1$. Then, the following are equivalent:
\begin{enumerate}
\item
${\rm T}_{n}(R,\alpha )$ is a $k$-$\Delta U$ ring.
\item
$R$ is a $k$-$\Delta U$ ring.
\end{enumerate}
\end{proposition}

\begin{proof}
Choose the set
	$$I:=\left\{
	\left.
	\begin{pmatrix}
		0 & a_{12} & \ldots & a_{1n} \\
		0 & 0 & \ldots & a_{2n} \\
		\vdots & \vdots & \ddots & \vdots \\
		0 & 0 & \ldots & 0
	\end{pmatrix} \right| a_{ij}\in R \quad (i\leq j )
	\right\}.$$
Then, one easily verifies that $I\subseteq J({\rm T}_{n}(R,\alpha ))$ and $\dfrac{{\rm T}_{n}(R,\alpha )}{I} \cong R$. Consequently, Theorem \ref{2.7} directly applies to get the expected result.
\end{proof}

A simple manipulation with coefficients guarantees that there is a ring isomorphism $$\varphi : \dfrac{R[x,\alpha]}{(x^n)}\rightarrow {\rm T}_{n}(R,\alpha),$$ given by $$\varphi (a_{0}+a_{1}x+\ldots +a_{n-1}x^{n-1}+\langle x^{n} \rangle )=(a_{0},a_{1},\ldots ,a_{n-1})$$ with $a_{i}\in R$, $0\leq i\leq n-1$. So, one finds that ${\rm T}_{n}(R,\alpha )\cong \dfrac{R[x,\alpha ]}{(x^n)}$, where $(x^n)$ is the ideal generated by $x^{n}$.

\medskip

We, thus, proceed by discovering the following two claims.

\begin{corollary}\label{3.11}
Let $R$ be a ring and $k\geq 1$. Then, the following are equivalent:
\begin{enumerate}
\item
$R$ is a $k$-$\Delta U$ ring.
\item
For $n \geq 2$, the quotient-ring $\dfrac{R[x; \alpha]}{(x^n)}$ is a $k$-$\Delta U$ ring.
\item
For $n \geq 2$, the quotient-ring $\dfrac{R[[x; \alpha]]}{(x^n)}$ is a $k$-$\Delta U$ ring.
\end{enumerate}
\end{corollary}

\begin{corollary}\label{3.12}
Let $R$ be a ring. Then, the following are equivalent:
\begin{enumerate}
\item
$R$ is a $k$-$\Delta$U ring.
\item
For $n \geq 2$, the quotient-ring $\dfrac{R[x]}{(x^n)}$ is a $k$-$\Delta U$ ring.
\item
For $n \geq 2$, the quotient-ring $\dfrac{R[[x]]}{(x^n)}$ is a $k$-$\Delta U$ ring.
\end{enumerate}
\end{corollary}

Consider now $R$ to be a ring and $M$ to be a bi-module over $R$. Let $${\rm DT}(R,M) := \{ (a, m, b, n) | a, b \in R, m, n \in M \}$$ with addition defined component-wise and multiplication defined by $$(a_1, m_1, b_1, n_1)(a_2, m_2, b_2, n_2) = (a_1a_2, a_1m_2 + m_1a_2, a_1b_2 + b_1a_2, a_1n_2 + m_1b_2 + b_1m_2 +n_1a_2).$$ Then, one claims that ${\rm DT}(R,M)$ is a ring which is isomorphic to ${\rm T}({\rm T}(R, M), {\rm T}(R, M))$. Also, we have $${\rm DT}(R, M) =
\left\{\begin{pmatrix}
	a &m &b &n\\
	0 &a &0 &b\\
	0 &0 &a &m\\
	0 &0 &0 &a
\end{pmatrix} |  a,b \in R, m,n \in M\right\}.$$ Likewise, one asserts that the following map is an isomorphism of rings: $\dfrac{R[x, y]}{\langle x^2, y^2\rangle} \rightarrow {\rm DT}(R, R)$, defined by $$a + bx + cy + dxy \mapsto
\begin{pmatrix}
	a &b &c &d\\
	0 &a &0 &c\\
	0 &0 &a &b\\
	0 &0 &0 &a
\end{pmatrix}.$$

We, thereby, detect the following.

\begin{corollary}\label{3.5}
Let $R$ be a ring and $M$ a bi-module over $R$. Then, the following statements are equivalent:
\begin{enumerate}
\item
$R$ is an $n$-$\Delta U$ ring.
\item
${\rm DT}(R, M)$ is an $n$-$\Delta U$ ring.
\item
${\rm DT}(R, R)$ is an $n$-$\Delta U$ ring.
\item
$\dfrac{R[x, y]}{\langle x^2, y^2\rangle}$ is an $n$-$\Delta U$ ring.
\end{enumerate}
\end{corollary}

Let $A$, $B$ be two rings and $M$, $N$ be $(A,B)$-bi-module and $(B,A)$-bi-module, respectively. Also, we consider the bi-linear maps $\phi :M\otimes_{B}N\rightarrow A$ and $\psi:N\otimes_{A}M\rightarrow B$ that apply to the following properties:
$$Id_{M}\otimes_{B}\psi =\phi \otimes_{A}Id_{M},Id_{N}\otimes_{A}\phi =\psi \otimes_{B}Id_{N}.$$
For $m\in M$ and $n\in N$, define $mn:=\phi (m\otimes n)$ and $nm:=\psi (n\otimes m)$. Now the $4$-tuple $R=\begin{pmatrix}
	A & M\\
	N & B
\end{pmatrix}$ becomes to an associative ring with obvious matrix operations that is called a {\it Morita context } ring. Denote two-side ideals $Im \phi$ and $Im \psi$ to $MN$ and $NM$, respectively, that are called the {\it trace ideals} of the {\it Morita context}.

\medskip

We now have at our disposal all the ingredients necessary to establish the following.

\begin{proposition}\label{3.14}
Let $R=\left(\begin{array}{ll}A & M \\ N & B\end{array}\right)$ be a Morita context ring. Then, $R$ is a (2n-1)-$\Delta U$ ring if, and only if, both $A$, $B$ are (2n-1)-$\Delta U$ and  $MN \subseteq J(A)$, $NM\subseteq J(B)$.
\end{proposition}
\begin{proof}
Let \( R \) be a $(2n-1)$-$\Delta$U ring. Consider \( e := \begin{pmatrix} 1_A & 0 \\ 0 & 1_B \end{pmatrix} \). Then, one says that \( eRe \cong A \) and \( (1 - e)R(1 - e) \cong B \).
So, thankfully to Proposition \ref{2.9}, we get that \( A, B \) are both $(2n-1)$-$\Delta$U. Obviously, \( \begin{pmatrix} 1 & m \\ 0 & 1 \end{pmatrix} \in U(R) \). Therefore, \[ \begin{pmatrix} 1 & m \\ 0 & 1 \end{pmatrix}^{2n-1} = \begin{pmatrix} 1 & (2n-1)m \\ 0 & 1 \end{pmatrix} \in 1+ \Delta(R) \] and hence \(\begin{pmatrix} 0 & (2n-1)m \\ 0 & 0 \end{pmatrix} \in \Delta(R) \). Similarly, we obtain \( \begin{pmatrix} 0 & 0 \\ (2n-1)m' & 0 \end{pmatrix} \in \Delta(R) \), where \( m' \in N \). Since \( 2 \in \Delta(R), 2n - 1 \in U(A) \), for any \( m \in M \) and \( m' \in N \) we receive that \[ \begin{pmatrix} (2n-1)^{-1} & 0 \\ 0 & 1 \end{pmatrix} \begin{pmatrix} 0 & (2n-1)m \\ 0 & 0 \end{pmatrix} \in \Delta(R). \]
Then, it must be that \( \begin{pmatrix} 0 & m \\ 0 & 0 \end{pmatrix} \in \Delta(R) \). Also,
\[\begin{pmatrix} 0 & 0 \\ (2n-1)m' & 0 \end{pmatrix} \begin{pmatrix} (2n-1)^{-1} & 0 \\ 0 & 1 \end{pmatrix}  \in \Delta(R). \]
Thus, \( \begin{pmatrix} 0 & 0 \\ m' & 0 \end{pmatrix} \in \Delta(R). \)
Since \( \Delta(R) \) is a subring, we have
\( \begin{pmatrix} 0 & M \\ N & 0 \end{pmatrix} \in \Delta(R) \). Then, for any \( m \in M \) and \( m' \in N \), we have
\[
\begin{pmatrix} 0 & m \\ 0 & 0 \end{pmatrix} \begin{pmatrix} 0 & 0 \\ m' & 0 \end{pmatrix} \in \Delta(R) \Rightarrow \begin{pmatrix} MN & 0 \\ 0 & 0 \end{pmatrix} \in \Delta(R),
\]
\[
\begin{pmatrix} 0 & 0 \\ m' & 0 \end{pmatrix} \begin{pmatrix} 0 & m \\ 0 & 0 \end{pmatrix} \in \Delta(R) \Rightarrow \begin{pmatrix} 0 & 0 \\ 0 & NM \end{pmatrix} \in \Delta(R).
\]
Since \( \Delta(R) \) is a subring, we can verify that \( I := \begin{pmatrix} MN & M \\ N & NM \end{pmatrix} \subseteq \Delta(R) \) and $I$ is an ideal, whence \( I \subseteq J(R) \). Consequently, \( MN \subseteq J(A) \) and \( NM \subseteq J(B) \) invoking \cite[Theorem 2.5]{13}, as required.

Reciprocally, let \( A, B \) be $(2n-1)$-$\Delta$U, where \( MN \subseteq J(A) \) and \( NM \subseteq J(B) \). Then, utilizing \cite[Lemma 3.1]{13}, we derive that \(J(R) = \begin{pmatrix} J(A) & M \\ N & J(B) \end{pmatrix} \). Thus, the isomorphism \( \frac{R}{J(R)} \cong \frac{A}{J(A)} \times \frac{B}{J(B)} \) is fulfilled. Finally, that \( R \) is $(2n-1)$-$\Delta$U is guaranteed by virtue of Proposition \ref{2.3} and Corollary \ref{2.8}, as needed.
\end{proof}

The next comments are worthwhile.

\begin{remark}
Exploiting Proposition \ref{3.14}, we have that if \( R \) is $(2n)$-$\Delta$U, then both \( A, B \) are $(2n)$-$\Delta$U and the containments \( (2n)MN \subseteq J(A) \), \( (2n)NM \subseteq J(B) \) hold. Now, a quite logical question arises that, if \( A, B \) are $(2n)$-$\Delta$U, where \( (2n)MN \subseteq J(A) \) and \( (2n)NM \subseteq J(B) \), can it be deduced that \( R \) is a $(2n)$-$\Delta$U ring?\\
However, the answer is negative as the following construction illustrates: letting \( R := \mathbb{F}_2 \langle x, y | x^2 = 0 \rangle \), then it can be checked that \( R \) is $2$-$\Delta$U and \( 2R = \{0\} \), but \( M_2(R) \) is not $2$-$\Delta$U.

Moreover, an other natural question arises, namely that if \( R \) is a $(2n)$-$\Delta$U ring, whether it be derived that \( MN \subseteq J(A) \) and \( NM \subseteq J(B) \)?\\ Again, the answer is contrapositive, because we know that \( M_2(\mathbb{Z}_2) \) is $6$-$\Delta$U; in fact, supposing \( A = B = M = N = \mathbb{Z}_2 \), then \( R = M_2(\mathbb{Z}_2) \) is $6$-$\Delta$U, but \( MN \not\subseteq J(A) \) and \( NM \not\subseteq J(B) \), as it can be verified without any difficulty.
\end{remark}

The following result could also be of some helpfulness and importance.

\begin{proposition}\label{3.6}
Let $R=\left(\begin{array}{ll}A & M \\ N & B\end{array}\right)$ be a Morita context ring such that $MN \subseteq J(A)$ and $NM\subseteq J(B)$. Then, $R$ is an $n$-$\Delta U$ ring if, and only if, both $A$ and $B$ are $n$-$\Delta U$.
\end{proposition}

\begin{proof}
In view of \cite [Lemma 3.1]{13}, we argue that $$J(R)=\begin{pmatrix}
		J(A) & M \\
		N & J(B)
\end{pmatrix}$$ and hence the isomorphism $\dfrac{R}{J(R)}\cong \dfrac{A}{J(A)}\times \dfrac{B}{J(B)}$ holds. Then, the result follows immediately from Corollary \ref{2.8} and Proposition \ref{2.3}.
\end{proof}

Now, let $R$, $S$ be two rings, and let $M$ be an $(R,S)$-bi-module such that the operation $(rm)s = r(ms$) is valid for all $r \in R$, $m \in M$ and $s \in S$. Given such a bi-module $M$, we can set

$$
{\rm T}(R, S, M) =
\begin{pmatrix}
	R& M \\
	0& S
\end{pmatrix}
=
\left\{
\begin{pmatrix}
	r& m \\
	0& s
\end{pmatrix}
: r \in R, m \in M, s \in S
\right\},
$$
where it forms a ring with the usual matrix operations. The so-stated formal matrix ${\rm T}(R, S, M)$ is called a {\it formal triangular matrix ring}. In Proposition \ref{3.6}, if we set $N =\{0\}$, then we will obtain the following.

\begin{corollary}\label{3.7}
Let $R,S$ be rings and let $M$ be an $(R,S)$-bi-module. Then, the formal triangular matrix ring ${\rm T}(R,S,M)$ is an $n$-$\Delta U$ ring if, and only if, both $R$ and $S$ are $n$-$\Delta U$.
\end{corollary}

Given a ring $R$ and a central element $s$ of $R$, the $4$-tuple $\begin{pmatrix}
	R & R\\
	R & R
\end{pmatrix}$ becomes a ring with addition component-wise and with multiplication defined by
$$\begin{pmatrix}
	a_{1} & x_{1}\\
	y_{1} & b_{1}
\end{pmatrix}\begin{pmatrix}
	a_{2} & x_{2}\\
	y_{2} & b_{2}
\end{pmatrix}=\begin{pmatrix}
	a_{1}a_{2}+sx_{1}y_{2} & a_{1}x_{2}+x_{1}b_{2} \\
	y_{1}a_{2}+b_{1}y_{2} & sy_{1}x_{2}+b_{1}b_{2}
\end{pmatrix}.$$
This ring is denoted by $K_s(R)$. A {\it Morita context}
$\begin{pmatrix}
	A & M\\
	N & B
\end{pmatrix}$ with $A=B=M=N=R$ is called a {\it generalized matrix ring} over $R$. It was observed by Krylov in \cite{15} that a ring $S$ is a generalized matrix ring over $R$ if, and only if, $S=K_s(R)$ for some $s\in C(R)$. Here, $MN=NM=sR$, so $MN\subseteq J(A)\Longleftrightarrow s\in J(R)$, $NM\subseteq J(B)\Longleftrightarrow s\in J(R)$.

\medskip

We, thus, have all the instruments to prove the following.

\begin{corollary}\label{3.8}
Let $R$ be a ring and $s\in C(R)\cap J(R)$. Then, $K_s(R)$ is an $n$-$\Delta U$ ring if, and only if, $R$ is $n$-$\Delta U$.
\end{corollary}

Following Tang and Zhou (cf. \cite{14}), for $n\geq 2$ and for $s\in C(R)$, the $n\times n$ {\it formal matrix ring} over $R$, defined with the usage of $s$ and denoted by $M_{n}(R;s)$, is the set of all $n\times n$ matrices over $R$ with the usual addition of matrices and with the multiplication defined below:

\medskip

\noindent For $(a_{ij})$ and $(b_{ij})$ in ${\rm M}_{n}(R;s)$,
$$(a_{ij})(b_{ij})=(c_{ij}), \quad \text{where} ~~ (c_{ij})=\sum s^{\delta_{ikj}}a_{ik}b_{kj}.$$
Here, $\delta_{ijk}=1+\delta_{ik}-\delta_{ij}-\delta_{jk}$, where $\delta_{jk}$, $\delta_{ij}$, $\delta_{ik}$ are the standard {\it Kroncker} delta symbols.

\medskip

We now offer the validity of the following.

\begin{corollary}\label{3.9}
Let $R$ be a ring and $s\in C(R)\cap J(R)$. Then, for any $k\geq 1$, $M_{n}(R;s)$ is a $k$-$\Delta U$ ring if, and only if, $R$ is $k$-$\Delta U$.
\end{corollary}

\begin{proof}
If $n = 1$, then ${\rm M}_n(R;s) = R$. So, in this case, there is nothing to prove. Let $n=2$. By the definition of ${\rm M}_n(R;s)$, we have ${\rm M}_2 (R;s) \cong {\rm K}_{s^2} (R)$. Apparently, $s^2 \in J(R) \cap C(R)$, so the claim holds for $n = 2$ with the help of Corollary \ref{3.8}.
	
To proceed by induction, assume now that $n>2$ and that the claim holds for ${\rm M}_{n-1} (R;s)$. Set $A := {\rm M}_{n-1} (R;s)$. Then, ${\rm M}_n (R;s) =
\begin{pmatrix}
		A & M\\
		N & R
\end{pmatrix}$
is a {\it Morita context}, where $$M =
\begin{pmatrix}
		M_{1n}\\
		\vdots\\
		M_{n-1, n}
\end{pmatrix}
\quad \text{and} \quad  N = (M_{n1} \dots M_{n, n-1})$$ with $M_{in} = M_{ni} = R$ for all $i = 1, \dots, n-1,$ and
\begin{align*}
		&\psi: N \otimes M \rightarrow N, \quad n \otimes m \mapsto snm\\
		&\phi : M \otimes N \rightarrow M, \quad  m \otimes n \mapsto smn.
\end{align*}
Besides, for $x =
\begin{pmatrix}
		x_{1n}\\
		\vdots\\
		x_{n-1, n}
\end{pmatrix}
	\in M$ and $y = (y_{n1} \dots y_{n, n-1}) \in N$, we write $$xy =
\begin{pmatrix}
		s^2x_{1n}y_{n1} & sx_{1n}y_{n2} & \dots & sx_{1n}y_{n, n-1}\\
		sx_{2n}y_{n1} & s^2x_{2n}y_{n2} & \dots & sx_{2n}y_{n, n-1}\\
		\vdots & \vdots &\ddots & \vdots\\
		sx_{n-1, n}y_{n1} & sx_{n-1, n}y_{n2} & \dots & s^2x_{n-1, n}y_{n, n-1}
\end{pmatrix} \in sA$$ and $$yx = s^2y_{n1}x_{1n} + s^2y_{n2}x_{2n} + \dots + s^2y_{n, n-1}x_{n-1, n} \in s^2 R.$$ Since $s \in J(R)$, we see that $MN \subseteq J(A)$ and $NM\subseteq J(A)$. Thus, we obtain that $$\frac{{\rm M}_n (R; s)}{J({\rm M}_n (R; s))} \cong \frac{A}{J (A)} \times \frac{R}{J (R)}.$$ Finally, the induction hypothesis and Proposition \ref{3.6} yield the claim after all.
\end{proof}

A {\it Morita context} $\begin{pmatrix}
	A & M\\
	N & B
\end{pmatrix}$ is called {\it trivial} if the context products are trivial, i.e., $MN=(0)$ and $NM=(0)$. Consulting with \cite{20}, we now are able to state that
$$\begin{pmatrix}
	A & M\\
	N & B
\end{pmatrix}\cong {\rm T}(A\times B, M\oplus N),$$
where
$\begin{pmatrix}
	A & M\\
	N & B
\end{pmatrix}$ is a trivial Morita context. We, therefore, begin the proof-check of the following.

\begin{corollary}\label{3.10}
The trivial Morita context
$\begin{pmatrix}
		A & M\\
		N & B
\end{pmatrix}$ is an $n$-$\Delta U$ ring if, and only if, both $A$ and $B$ are $n$-$\Delta U$.
\end{corollary}

\begin{proof}
It is apparent to see that the two isomorphisms
	$$\begin{pmatrix}
		A & M\\
		N & B
	\end{pmatrix} \cong {\rm T}(A\times B,M\oplus N) \cong \begin{pmatrix}
		A\times B & M\oplus N\\
		0 & A \times B
	\end{pmatrix}$$ are true. Then, the rest of the proof follows by combining Proposition \ref{3.4}(i) and \ref{2.3}, as needed.
\end{proof}

As usual, for an arbitrary ring $R$ and an arbitrary group $G$, the symbol $RG$ stands for the {\it group ring} of $G$ over $R$. Standardly, $\varepsilon(RG)$ designates the kernel of the classical {\it augmentation map} $\varepsilon: RG\to R$, defined by $\varepsilon (\displaystyle\sum_{g\in G}a_{g}g)=\displaystyle\sum_{g\in G}a_{g}$, and this ideal is called the {\it augmentation ideal} of $RG$.

\medskip

Here we will explore group rings that are $n$-$\Delta U$, as for the case of JU group rings we refer the interested reader to \cite{KZ}. Specifically, we continue by establishing the next three technicalities.

\begin{lemma}
If $RG$ is an $n$-\(\Delta U\) ring, then $R$ is too $n$-\(\Delta U\).
\end{lemma}

\begin{proof}
Choosing $u \in U(R)$, then $u \in U(RG)$. Thus, $u^n=1+r$, where $r \in \Delta(RG)$. Since $r=1-u^n \in R$, it suffices to show that $r \in \Delta(R)$, which is obviously true, because, for any $v \in U(R) \subseteq U(RG)$, we have $v-r \in U(RG) \cap R \subseteq U(R)$. Therefore, $r \in \Delta(R)$, as required.
\end{proof}

We say that a group $G$ is a {\it $p$-group} if every element of $G$ is a power of the prime number $p$. Besides, a group $G$ is said to be {\it locally finite} if every finitely generated subgroup is finite.

\begin{lemma}\label{4.14} \cite[Lemma $2$]{16}.
Let $p$ be a prime with $p\in J(R)$. If $G$ is a locally finite $p$-group, then $\varepsilon(RG) \subseteq J(RG)$.
\end{lemma}

\begin{lemma}
If \( R \) is an $n$-\(\Delta U\) ring and \( G \) is a locally finite \( p \)-group, where \( p \) is a prime number such that \( p \in J(R) \), then \( RG \) is an $n$-\(\Delta U\) ring.
\end{lemma}

\begin{proof}
One looks that Lemma \ref{4.14} tells us that $\varepsilon(RG) \subseteq J(RG)$. On the other hand, since the isomorphism \( RG/\varepsilon(RG) \cong R \) holds, Theorem \ref{2.7} is a guarantor that \( RG \) is an $n$-\(\Delta U\) ring, as stated.
\end{proof}

%\noindent{\bf Acknowledgement.} The authors express their sincere gratitude to the expert referee for the numerous competent suggestions made which lead to a substantial improvement of the exposition.

\medskip

\noindent{\bf Funding:} The work of the first-named author, P.V. Danchev, is partially supported by the project Junta de Andaluc\'ia under Grant FQM 264. All other four authors are supported by Bonyad-Meli-Nokhbegan and receive funds from this foundation.


\vskip5.0pc

\begin{thebibliography}{99}
	
\bibitem{10}
A. Badawi, {\it On abelian $\pi$-regular rings}, Commun. Algebra {\bf 25}(4) (1997), 1009–-1021.

\bibitem{21}
W. Chen, {\it On constant products of elements in skew polynomial rings}, Bull. Iran. Math. Soc. {\bf 41}(2)(2015), 453--462.

\bibitem{Danew}
P. V. Danchev, {\it Rings with Jacobson units}, Toyama Math. J. {\bf 38}(1) (2016), 61--74.

\bibitem{4}
P. V. Danchev, {\it On exchange $\pi$-UU unital rings}, Toyama Math. J. {\bf 39}(1) (2017), 1--7.

\bibitem{Dannew}
P. V. Danchev, {\it On exchange $\pi$-JU unital rings}, Rend. Sem. Mat. Univ. Pol. Torino {\bf 77}(1) (2019), 13--23.

\bibitem{19}
P. V. Danchev, A. Javan and A. Moussavi, {\it Rings with $u^n-1$ nilpotent for each unit $u$}, J. Algebra Appl. {\bf 25} (2026).

\bibitem{11}
A. J. Diesl, {\it Nil clean rings}, J. Algebra {\bf 383} (2013), 197--211.	
	
\bibitem{1}
F. Karabacak, M. T. Ko\c{s}an, T. Quynh and D. Tai, {\it A generalization of UJ-rings}, J. Algebra Appl. {\bf 20} (2021).

\bibitem{KLM}
M. T. Ko\c{s}an, A. Leroy and J. Matczuk, {\it On UJ-rings}, Commun. Algebra {\bf 46}(5) (2018), 2297--2303.

\bibitem{3}
M. T. Ko\c{s}an, T. C. Quynh, T. Yildirim and J. $\check{Z}$emli$\check{c}$ka, {\it Rings such that, for each unit $u$, $u - u^n$ belongs to the Jacobson radical}, Hacettepe J. Math. Stat. {\bf 49}(4) (2020), 1397--1404.

\bibitem{20}
M. T. Ko\c{s}an, Y. C. Quynh and J. $\check{Z}$emli$\check{c}$ka, {\it UNJ-Rings}, J. Algebra Appl. {\bf 19} (2020).

\bibitem{KZ}
M. T. Ko\c{s}an and J. $\check{Z}$emli$\check{c}$ka, {\it Group rings that are UJ rings}, Commun. Algebra {\bf 49}(6) (2021),  2370--2377.

\bibitem{15}
P. A. Krylov, {\it Isomorphism of generalized matrix rings}, Algebra Logic {\bf 47} (4) (2008), 258-262.	

\bibitem{12}
T. Y. Lam, {\it A First Course in Noncommutative Rings}, Second Edition, Springer Verlag, New York, 2001.

\bibitem{18}
T. Y. Lam, {\it Exercises in Classical Ring Theory}, Second Edition, Springer Verlag, New York, 2003.

\bibitem{2}
A. Leroy and J. Matczuk, {\it Remarks on the Jacobson radical}, Rings, Modules and Codes, Contemp. Math., {\bf 727}, Am. Math. Soc., Providence, RI, 2019, 269--276.

\bibitem{9}
J. Levitzki, {\it On the structure of algebraic algebras and related rings}, Trans. Am. Math. Soc. {\bf 74} (1953), 384--409.

\bibitem{17}
A. R. Nasr-Isfahani, {\it On skew triangular matrix rings}, Commun. Algebra {\bf 39}(11) (2011), 4461--4469.

\bibitem{5}
W. K. Nicholson, {\it Lifting idempotents and exchange rings}, Trans. Am. Math. Soc. {\bf 229} (1977), 269-–278.		 

\bibitem{13}
G. Tang, C. Li and Y. Zhou, {\it Study of Morita contexts}, Commun. Algebra {\bf 42} (4) (2014), 1668--1681.

\bibitem{14}
G. Tang and Y. Zhou, {\it A class of formal matrix rings}, Linear Algebra Appl. {\bf 438} (2013), 4672--4688.

\bibitem{16}
Y. Zhou, {\it On clean group rings}, Advances in Ring Theory, Treads in Mathematics, Birkhauser, Verlag Basel/Switzerland, 2010, 335--345.

\end{thebibliography}
\end{document}